\documentclass[amsfonts]{amsart}

\usepackage{amssymb,amsfonts}
\usepackage[all,arc]{xy}
\usepackage{enumerate}
\usepackage{mathrsfs}
\usepackage[autostyle]{csquotes}
\usepackage{amscd}
\newtheorem{thm}{Theorem}[section]

\newtheorem{prop}[thm]{Proposition}

\theoremstyle{definition}
\newtheorem{defn}[thm]{Definition}

\theoremstyle{remark}

\makeatletter
\let\c@equation\c@thm
\makeatother
\numberwithin{equation}{section}

\title{Effective Bounds on Topological Types of Real Algebraic and Semialgebraic Sets}

\author{Kartoue Mady Demdah$^{\star}$}

\address{
\begin{quote}
$^{\star}$AI Team, Olameter Inc.
\newline 2000 McGill College Ave.
\newline Suite 500, Montreal, Quebec H3A 3H3
\newline CANADA
\newline \emph{Email address: kartoue@gmail.com}
\end{quote}
}

\author{Ibrahim Nonkane$^{\dagger}$}
\address{
		\begin{quote}
		$^{\dagger}$Department of Mathematics
		\newline University Thomas Sankara
		\newline BURKINA FASO
		\newline \emph{Email address: nonkane\_ibrahim@yahoo.fr}
		\end{quote}}


\date{}

\begin{document}

\maketitle
\begin{abstract}
In this paper, we revisit the problem of classifying real algebraic and semialgebraic sets by their topological types, focusing on establishing the effectiveness of bounds rather 
than deriving new quantitative estimates. Building on Hardt's theorem and leveraging a model-theoretic framework,  we prove the existence of effective  bounds on the number of distinct topological types for real algebraic sets defined by polynomials of bounded degree. While less precise than previously known doubly exponential bounds derived from methods such as Stratified 
Cylindrical Algebraic Decomposition, our results are obtained with minimal additional work, emphasizing simplicity and uniformity. 

Additionally, we establish the existence of an effective bound on the complexity of the semialgebraic homeomorphisms that describe representatives of the topological equivalence classes, offering a novel perspective 
on their explicit classification. These findings complement existing results by demonstrating the effectiveness of uniform bounds within a logical and geometric framework. 
The paper concludes with an extension of these results to semialgebraic sets of bounded complexity, further emphasizing the uniformity of the approach across varying real closed 
fields.

\end{abstract}


\section*{Introduction}

The study of real algebraic sets, defined as the solution sets of polynomial equations over the real numbers, is a cornerstone of real algebraic geometry. 
A central challenge in this field is to understand the topological complexity of these sets and classify them up to homeomorphism, providing a comprehensive 
understanding of their structural diversity. This classification problem is fundamental both to theoretical mathematics and to practical applications in optimization, 
robotics, and data analysis, where real algebraic and semialgebraic sets naturally arise.

Classical results, such as the Oleinik-Petrovsky, Thom, and Milnor bounds, have offered profound insights into the global topological complexity of real algebraic sets 
by estimating quantities such as Betti numbers or Euler characteristics (cf. \cite{OP}, \cite{T}, \cite{BR}, \cite{BPR}). However, these aggregate measures of complexity 
do not directly address a finer question: how many distinct topological types can occur within a given family of real algebraic or semialgebraic sets? This question has been partially answered in work such as that of Basu, Pollack, and Roy, who noted in Algorithms in Real Algebraic Geometry (cf. \cite{BPR}, Remark 11.46.c) that the bound for the number of topological types is polynomial in d and doubly exponential in the number of monomials of degree d in k variables, derived using  their Stratified Cylindrical Algebraic Decomposition. Similarly, Basu and Vorobjov, in their work on homotopy types, highlight the doubly exponential bounds for topological types while obtaining simply exponential bounds for the number of homotopy types.

In this paper, we revisit the problem of bounding the number of topological types, adopting a distinct approach based on Hardt's theorem. By leveraging the model-theoretic 
framework of real closed fields and tools such as quantifier elimination, we demonstrate that the number of topological types of real algebraic sets defined by polynomials 
of bounded degree d can be effectively bounded. While the bounds established in this work are only recursive, and thus considerably less precise than the doubly exponential 
bounds cited above, they are derived almost “without further work” from Hardt's theorem, emphasizing the simplicity and elegance of this approach. Furthermore, we extend 
these results by providing a recursive bound on the description of the semialgebraic homeomorphisms representing each topological equivalence class.

The main contributions of this paper are twofold: first, we establish effective bounds on the number of topological types of real algebraic and semialgebraic sets of bounded complexity; and second, we derive the existence of an effective bound on the complexities of the semialgebraic homeomorphisms associated with the representatives of these topological equivalence classes. 

The structure of this article is as follows. In Section 1, we provide the necessary background on semialgebraic geometry and the model theory of real closed fields, with a detailed exposition of the Tarski-Seidenberg theorem. Section 2 establishes the effectiveness of bounds on the number of topological types of real algebraic sets defined by polynomials with a bounded degree. In addition, we extend these results by deriving the existence of an effective bound on the complexities of the semialgebraic homeomorphisms associated with the representatives of these topological equivalence classes.  Section 3 extends these results to semialgebraic sets of bounded complexity, demonstrating that these results hold uniformly across varying real closed fields.

\section{Preliminaries and Theoretical Framework}

This section introduces the basic concepts from real algebraic geometry, semialgebraic sets, and model theory of real closed fields, which provide the theoretical framework for analyzing the topological complexity of real algebraic sets. We begin by defining algebraic and semialgebraic sets, present essential results such as the Tarski-Seidenberg theorem, and  Tarski's decidability theorem.

We start by recalling the definitions of real algebraic and semialgebraic sets, which form the core objects of our study.

\begin{defn} \rm
A \emph{real algebraic set} \( Z \subseteq \mathbb{R}^n \) is the zero locus of a finite collection of polynomials \( f_1, \ldots, f_m \in \mathbb{R}[X_1, \ldots, X_n] \):
\[
Z = \{ x \in \mathbb{R}^n \mid f_1(x) = 0, \ldots, f_m(x) = 0 \}.
\]
The \emph{degree} of \( Z \) is defined as the maximum of the degrees of the polynomials \( f_1, \ldots, f_m \).
\end{defn}

We recall in the following the concept of semialgebraic subset and a notion of complexity for semialgebraic sets, which associates to each such set a pair of integers that provide information about the number of polynomials defining the set and the maximal degree of these polynomials. Formally, we define:
\begin{defn}\rm
Let \( R \) be a real closed field.
\begin{itemize}
    \item A set \( S \subseteq R^n \) is called \emph{semialgebraic} if it can be expressed as a finite Boolean combination of sets defined by polynomial equations and inequalities. Specifically, \( S \) can be written as
    \[
    S = \bigcup_{i=1}^s \bigcap_{j=1}^{k_i} \{ x \in R^n \mid f_{ij}(x) \ast_{ij} 0 \},
    \]
    where \( f_{ij} \in R[X_1, \ldots, X_n] \) are polynomials, and each \( \ast_{ij} \) is either \( = \) or \( > \), for \( i = 1, \ldots, s \) and \( j = 1, \ldots, k_i \).

    \item A semialgebraic set \( S \subseteq R^n \) is said to have \emph{complexity} at most \( (p, q) \) if it can be described using at most \( p \) polynomial conditions (i.e., \( \sum_{i=1}^s k_i \leq p \)) and all the polynomials \( f_{ij} \) have degree at most \( q \).

    \item The \emph{complexity} of a semialgebraic set \( S \subseteq R^n \) is defined as the minimal pair \( (p, q) \) (with respect to the lexicographic order) such that \( S \) admits a description of the form above, with at most \( p \) polynomial conditions and polynomials of degree at most \( q \).
\end{itemize}

\end{defn}

We now turn to some basic notions from the model theory of real closed fields, which will be crucial for establishing the uniformity and effectiveness results in our subsequent analysis.

\begin{defn} \rm
A \emph{real closed field} \( R \) is an ordered field satisfying two key properties:
\begin{enumerate}
	\item Every positive element in \( R \) has a square root.
	\item Every polynomial of odd degree with coefficients in \( R \) has a root in \( R \).
\end{enumerate}
The field of real numbers \( \mathbb{R} \) is the canonical example of a real closed field.
\end{defn}

To present Hardt's Semialgebraic Trivialization Theorem, we first revisit the concept of semialgebraic trivialization of a semialgebraic continuous map.
\begin{defn}\rm
Let \( f: A \rightarrow B \) be a continuous semialgebraic map. We say that \( f \) is \emph{semialgebraically trivial over a semialgebraic subset} \( C \subset B \) if there exist a semialgebraic set \( F \) and a semialgebraic homeomorphism
\(
h: f^{-1}(C) \rightarrow C \times F
\)
such that the composition of \( h \) with the projection map \( \mathrm{pr}_1 : C \times F \rightarrow C \) is equal to the restriction of \( f \) to \( f^{-1}(C) \). In other words, \( \mathrm{pr}_1 \circ h = f|_{f^{-1}(C)} \).

\end{defn}
We now state Hardt’s Semialgebraic Trivialization Theorem. A comprehensive proof valid over any real closed field is provided in [BCR, p. 221].

\begin{thm}[\textbf{Hardt’s Semialgebraic Trivialization Theorem}] 
Let \( A \subset \mathbb{R}^n \) and \( B \subset \mathbb{R}^m \) be semialgebraic sets, and let \( f: A \rightarrow B \) be a semialgebraic map. Then there exists a finite semialgebraic partition of \( B \) into semialgebraic sets \( B_1, \ldots, B_k \) such that \( f \) is semialgebraically trivial over each \( B_i \).

Moreover, given finitely many semialgebraic subsets \( A_1, \ldots, A_h \subset A \), we can choose the trivializations \( h_i \) over \( B_i \) to be compatible with all the subsets \( A_j \).

\end{thm}

A fundamental result in real algebraic geometry is the \emph{Tarski-Seidenberg theorem}, which ensures that the class of semialgebraic sets is closed under projections. This theorem plays a crucial role in our study, allowing us to reduce high-dimensional problems to lower-dimensional ones while preserving semialgebraicity.

\begin{thm}[\textbf{Tarski-Seidenberg Theorem}]

Let \( S \subseteq \mathbb{R}^{k+m} \) be a semialgebraic set. Then its projection onto \( \mathbb{R}^k \),
\[
\pi(S) = \{ x \in \mathbb{R}^k \mid \exists y \in \mathbb{R}^m \, \text{ such that } (x, y) \in S \},
\]
is also a semialgebraic set.
\end{thm}

The theorem below presents the Transfer Principle, which follows from the Tarski-Seidenberg Theorem (see [BCR, Proposition 5.2.3] for details). We restate it here as follows.

\begin{thm} \label{Principle}
Let \(\Phi\) be a first-order formula in the theory of real closed fields with coefficients in \(\mathbb{Z}\). Then \(\mathbb{R} \models \Phi\) if and only if \(R \models \Phi\) for any real closed field \(R\).
\end{thm}

We conclude this section by recalling a result of Tarski on the decidability of the theory of real closed fields. This result provides a solid foundation for the effective analysis of semialgebraic sets in a model-theoretic context. For a detailed proof, please refer to Theorem 6.4 in \cite{P}.

\begin{thm}[\textbf{Tarski's Decidability Theorem}]\label{decidability}
The first-order theory of real closed fields is decidable. In particular, there is an algorithm that, given any sentence \( \sigma \) in the language of real closed fields, determines whether \( \sigma \) is true in the field of real numbers \( \mathbb{R} \).
\end{thm}

This theorem allows us to algorithmically verify properties of semialgebraic sets, making it a powerful tool for establishing the effectiveness of various bounds in real algebraic geometry.

\section{Effectiveness of the Number of Topological Types for Real Algebraic Sets}
In this section, we aim to establish effective bounds on two key aspects: (1) the number of topological types of algebraic sets defined by polynomial equations of bounded degree, and (2) the complexities of semialgebraic homeomorphisms between such algebraic sets. Specifically, we prove that for any positive integers \( n \) and \( d \), there exist computable bounds \( p(n, d) \), \( t(n, d) \), and \( u(n, d) \) such that every algebraic subset of \( \mathbb{R}^n \) defined by polynomials of degree at most \( d \) can be semialgebraically classified using a finite list of model sets \( V_1, \ldots, V_p \) and corresponding homeomorphisms whose complexities are bounded by \( (t(n, d), u(n, d)) \).

To achieve this, we first recall essential definitions related to effectiveness and uniform definability in the theory of real closed fields. We then present key propositions that provide the foundational tools necessary for proving our main result.

\begin{defn}\rm
A family of first-order formulas \(\mathcal{F}(n_1, \ldots, n_r)\) in the theory of real closed fields, indexed by tuples \((n_1, \ldots, n_r) \in \mathbb{N}^r\), is called effective if there exists an algorithm that, given \( n_1, \ldots, n_r \in \mathbb{N} \), constructs the corresponding formula \(\mathcal{F}(n_1, \ldots, n_r)\).
\end{defn}

\begin{defn}\rm
A semialgebraic set \( S \) is said to be \emph{uniformly defined} if there exists a first-order formula in the theory of real closed fields that defines \( S \) independently of the choice of a real closed field. In other words, \( S \) can be defined over any real closed field using the same formula.
\end{defn}

We recall the following result, which establishes the existence of a uniform bound \((t, u)\). Specifically, for any pair of semialgebraic sets with complexity at most \((p, q)\) that are semialgebraically homeomorphic, there exists a semialgebraic homeomorphism between them with complexity bounded by \((t, u)\). For a detailed proof, see Proposition 3.6 in \cite{D}.

\begin{prop}\label{sdf}
Given the integers $n$, $p$ and $q$, there  exists a couple
of integers $(t,u)$ such that for every couple of semialgebraic subsets of  $R^n$ with complexity at most \((p, q)\), semialgebraically homeomorphic, there is a semialgebraic homeomorphism $f$ between them whose graph $\Gamma_f $ with complexity at most \((t, u)\).
\end{prop}

With a given bound on the degrees of the defining polynomials, a similar argument as in the previous result shows that the family of algebraic sets defined by polynomial equations of degree at most \( d \) can be parameterized by a semialgebraic set. This can be formally stated as follows:

\begin{prop}\label{sap}
Let \( n, d \) be positive integers. There exists a semialgebraic subset \( \mathcal{M}(n, d) \) in some affine space \( \mathbb{R}^{\alpha(n, d)} \) and a semialgebraic family \( \mathcal{K}(n, d) \subseteq \mathcal{M}(n, d) \times \mathbb{R}^n \) such that:

\begin{enumerate}
	\item \textbf{Fiber Property}:  
For every \( a \in \mathcal{M}(n, d) \), the fiber  
\[
V_a(n, d) = \{ x \in R^n \mid (a, x) \in \mathcal{K}(n, d) \}
\]
is an algebraic subset of \( R^n \) with complexity at most \( d \).

	\item \textbf{Uniform Representation}:  
For every algebraic subset \( V \subseteq R^n \) of complexity at most \( d \), there exists an element \( a \in \mathcal{M}(n, d) \) such that \( V = V_a(n, d) \).

	\item \textbf{Uniformity in Construction}:  
The sets \( \mathcal{M}(n, d) \) and \( \mathcal{K}(n, d) \) are defined uniformly by first-order formulas in the theory of real closed fields, without parameters, and can be effectively constructed from \( n, d \).
\end{enumerate}
\end{prop}

The following proposition can be seen as a special case of Proposition \ref{sdf}. It can be demonstrated using the semialgebraic parameterization of algebraic sets defined by polynomial equations of degrees at most \(d\) (Proposition \ref{sap}) and Hard't's Theorem.

\begin{prop}\label{has}
Given integers \( n \) and \( d \), there exists a pair of integers \((t, u)\) such that for any two algebraic sets in \(\mathcal{M}(n, d)\) that are semialgebraically homeomorphic, there is a semialgebraic homeomorphism \( f \) between them whose graph \(\Gamma_f \) belongs to a semialgebraic set of complexity at most \((t, u)\).
\end{prop}

We now present a theorem that establishes a finite classification for real algebraic sets. For a detailed discussion of this result, see Theorem 9.3.5 in \cite{BCR}.

\begin{thm}
For every pair of positive integers \( n, d \), there exist a number \( p = p(n, d) \) and algebraic subsets \( V_1, \ldots, V_p \subseteq R^n \) defined by polynomial equations of degrees at most \( d \), such that, for any algebraic subset \( V \subset R^n \) of degree at most \( d \), there exists \( i \in \{1, \ldots, p\} \) and a semialgebraic homeomorphism \( h: R^n \rightarrow R^n \) with \( h(V_i) = V \).
\end{thm}

Considering a bound on graph complexities of semialgebraic homeomorphisms, we refine the theorem as follows:

\begin{thm} \label{born}
For every pair of positive integers \( n, d \), there exist integers \( p = p(n, d), t = t(n, d), u = u(n, d) \) and algebraic subsets \( V_1, \ldots, V_p \subseteq R^n \) defined by polynomial equations of degrees at most \( d \), such that, for every algebraic subset \( V \subset R^n \) of degrees at most \( d \), there exists \( i \in \{1, \ldots, p\} \) and a semialgebraic homeomorphism \( h: R^n \rightarrow R^n \) of complexity at most \((t, u)\) satisfying \( h(V_i) = V \).
\end{thm}

\begin{thm} \label{main}
The bounds \( p(n, d), t(n, d), \) and \( u(n, d) \) can be computed effectively.
\end{thm}

\begin{proof}
Given \( n \) and \( d \), there exist positive integers \( p, t, \) and \( u \) such that the following statement, denoted by \( \Phi(n, d, p) \), holds over the real numbers \( \mathbb{R} \):

\begin{quote}
\enquote{\emph{There exist \( p \) algebraic subsets \( V_1, \ldots, V_p \subseteq \mathbb{R}^n \) defined by polynomial equations of degree at most \( d \) such that for any algebraic subset \( V \subset \mathbb{R}^n \) defined by polynomial equations of degree at most \( d \), there exist \( i \in \{1, \ldots, p\} \) and a semialgebraic homeomorphism \( h : \mathbb{R}^n \rightarrow \mathbb{R}^n \) of complexity at most \((t, u)\) with \( h(V_i) = V \).}}
\end{quote}
We want to express \( \Phi(n, d, p) \) as a first-order formula in the language of real closed fields. To do so, we rely on the following key results:
\begin{enumerate}

	\item \textbf{Uniform Parameterization of Algebraic Sets}:  
   By Proposition \ref{sap}, the set of algebraic subsets of \( R^n \) defined by polynomial equations of degree at most \( d \) can be parameterized by a semialgebraic set denoted by \( \mathcal{M}(n, d) \). This set \( \mathcal{M}(n, d) \) is defined uniformly by a first-order formula of the theory of real closed fields, without reference to any particular real closed field. Thus, it can be effectively constructed from \( n \) and \( d \).

	\item \textbf{Semialgebraic Homeomorphisms}:  
  The property of two algebraic sets being semialgebraically homeomorphic can also be expressed using a first-order formula. By Proposition \ref{has}, there exists a semialgebraic set, denoted by \(\mathcal{SH}(n, d, t, u)\), defined in \(\mathcal{M}(n, d) \times \mathcal{M}(n, d) \times \mathcal{A}(2n, t, u)\), such that for any \((a, b, h) \in \mathcal{M}(n, d) \times \mathcal{M}(n, d) \times \mathcal{A}(2n, t, u)\), the element \((a, b, h)\) is in \(\mathcal{SH}(n, d, t, u)\) if and only if \( h : \mathbb{R}^n \rightarrow \mathbb{R}^n \) is a semialgebraic homeomorphism satisfying \( h(V_a) = V_b \).

\end{enumerate}
Given the above observations, we can translate \( \Phi(n, d, p) \) into a first order formula of the theory of real closed fields:

\[
\Phi(n, d, p, t, u) := \exists (a_1, \ldots, a_p) \in \mathcal{M}(n, d)^p \, \forall b \in \mathcal{M}(n, d) \, \bigvee_{i=1}^p \, (a_i, b, h) \in \mathcal{SH}(n, d, t, u),
\]
where \( a_1, \ldots, a_p \) are \( p \) distinct algebraic sets in \( \mathcal{M}(n, d) \), and for any algebraic subset \( V_b \in \mathcal{M}(n, d) \), there is some \( a_i \) and a homeomorphism \( h \) of bounded complexity such that \( h(V_i) = V_b \).

Now, using the decidability of the theory of real closed fields (Tarski’s theorem), we can devise an algorithm to find appropriate values for \( p, t, \) and \( u \):

\begin{enumerate}
	\item \textbf{Initialization}: Set \( p := 1, \, t := 1, \, u := 1 \).

	\item \textbf{Incremental Search}: While \( \Phi(n, d, p, t, u) \) is not true in the theory of real closed fields, increment \( p, t, \) and \( u \) simultaneously:

\[
p := p + 1, \quad t := t + 1, \quad u := u + 1.
\]

	\item \textbf{Termination}: Since \( \Phi(n, d, p, t, u) \) is known to hold for some finite values (by Theorem 2.2), this algorithm must terminate.
\end{enumerate}

This process effectively determines a recursive function \( (n, d) \mapsto (p, t, u) \). Hence, the bounds \( p(n, d), t(n, d), \) and \( u(n, d) \) are effective, completing the proof.
\end{proof}

\section{Generalization of Effectiveness of the Number of Topological Types for Semi-Algebraic sets}

In this section, we demonstrate that the number of topological types of semialgebraic subsets of \(\mathbb{R}^n\) with complexity at most \((p, d)\) can be effectively bounded. This extends the Theorem \ref{main} and \ref{born} on the existence of effective bounds for the number of topological classes of algebraic subsets defined by polynomial equations of degree at most \(d\).

We begin by recalling a proposition that provides a semialgebraic parametrization for subsets of \(\mathbb{R}^n\) with maximal complexity \((p, q)\). A detailed proof can be found in [D, Proposition 3.2].

\begin{prop}\label{pas}
Let \( n, p, \) and \( q \) be positive integers. There exists a semi-algebraic subset \( \mathcal{A}(n, p, q) \) in some affine space \( \mathbb{R}^{\beta(n, p, q)} \) and a semi-algebraic family \( \mathcal{S}(n, p, q) \subseteq \mathcal{A}(n, p, q) \times \mathbb{R}^n \) such that:
\begin{enumerate}
	\item \textbf{Fiber Property}:
	
   For every \( a \in \mathcal{A}(n, p, q) \), the fiber
   \[
   S_a(n, p, q) = \{ x \in \mathbb{R}^n \mid (a, x) \in \mathcal{S}(n, p, q) \}
   \]
   is a semi-algebraic subset of \( \mathbb{R}^n \) with complexity at most \( (p, q) \).

	\item \textbf{Uniform Representation}:
	
   For every semi-algebraic subset \( S \subseteq \mathbb{R}^n \) of complexity at most \( (p, q) \), there exists an element \( a \in \mathcal{A}(n, p, q) \) such that \( S = S_a(n, p, q) \).

	\item  \textbf{Uniformity in Construction}:
   The set \( \mathcal{A}(n, p, q) \) and the family \( \mathcal{S}(n, p, q) \) are defined in a uniform way by first-order formulas of the theory of real closed fields, without parameters, which can be effectively constructed from \( n, p, q \).
\end{enumerate}

\end{prop}

We now establish the finiteness of topological types for semialgebraic sets with bounded complexity.

\begin{thm}\label{main-2}
Let \( n, p \) and \( q \) be positive integers. Then there are positive integers \( t = t(n, p, q), u = u(n, p, q), v = v(n, p, q)\) and a collection of semi-algebraic subsets \( S_1, \ldots, S_t \subseteq \mathbb{R}^n \) of complexity at most \( (p, q) \), such that for any semi-algebraic subset \( S \subseteq \mathbb{R}^n \)  of complexity at most \((p, q) \), there exists an index \( i \in \{1, \ldots, t\} \) and a semi-algebraic homeomorphism \( h : \mathbb{R}^n \rightarrow \mathbb{R}^n \) of complexity at most \(u, v\) such that \( h(S_i) = S \).
\end{thm}

\begin{proof}
By Proposition \ref{pas}, there exists a semialgebraic subset \( \mathcal{A}(n, p, q) \) in some affine space \( \mathbb{R}^{\beta(n, p, q)} \) and a semialgebraic family \( \mathcal{S}(n, p, q) \subseteq \mathcal{A}(n, p, q) \times \mathbb{R}^n \) such that 
\[
S(n, p, q) = \{ (a, x) \in \mathcal{A}(n, p, q) \times \mathbb{R}^n : x \in S_a \},
\]
where \( S_a \) is a semialgebraic subset of \( \mathbb{R}^n \) parametrized by \( a \in \mathcal{A}(n, p, q) \). The set \( S(n, p, q) \) is therefore a semialgebraic subset of \( \mathbb{R}^{\alpha(n, p, q)} \times \mathbb{R}^n \).

Consider the projection map:
\[
\Pi: \mathbb{R}^{\beta(n, p, q)} \times \mathbb{R}^n \rightarrow \mathbb{R}^{\alpha(n, p, q)}, \quad (a, x) \mapsto a.
\]
The restriction of \( \Pi \) to \( S(n, p, q) \),
\[
\Pi|_{S(n, p, q)} : S(n, p, q) \rightarrow \mathcal{A}(n, p, q),
\]
is a semialgebraic map. By applying the Hardt trivialization theorem to \( \Pi|_{S(n, p, q)} \), we obtain a finite semialgebraic partition of \( \mathcal{A}(n, p, q) \) into semialgebraic subsets \( S_1, \ldots, S_t \) such that \( \Pi|_{S(n, p, q)} \) is semialgebraically trivial over each \( S_i \) for \( i = 1, \ldots, t \). This establishes the existence of a finite integer \( t \) corresponding to the number of semialgebraic pieces.

The existence of the integers \( u \) and \( v \) follows from Proposition \ref{sdf}, which completes the proof.

\end{proof}

We conclude this section by generalizing the previous theorem to any closed real field \(R\), establishing that the number of topological types of algebraic sets defined by polynomial equations of degree at most \(d\) can be effectively bounded.

\begin{thm} Le \(R\) be a real closed field.
\begin{enumerate}
	\item Let \( n, p \) and \( q \) be positive integers. Then there are positive integers \( t = t(n, p, q), u = u(n, p, q), v = v(n, p, q)\) and a collection of semi-algebraic subsets \( S_1, \ldots, S_t \subseteq R^n \) of complexity at most \( (p, q) \), such that for any semi-algebraic subset \( S \subseteq R^n \)  of complexity at most \((p, q) \), there exists an index \( i \in \{1, \ldots, t\} \) and a semi-algebraic homeomorphism \( h : R^n \rightarrow R^n \) of complexity at most \(u, v\) such that \( h(S_i) = S \).
	\item The bounds \( t = t(n, p, q), u = u(n, p, q), v = v(n, p, q)\) are effective.

\end{enumerate}

\end{thm}

\begin{proof}
Given \( n, p, \) and \( q \), there exist integers \( t, u, \) and \( v \in \mathbb{N} \) such that the following statement \(\Psi(n, d, p, t, u)\) holds over \(\mathbb{R}\):

\begin{quote}
\emph{There are \( t \) semialgebraic subsets \( S_1, \ldots, S_t \subseteq \mathbb{R}^n \) of complexity at most \( (p, q) \), such that for any semialgebraic subset \( S \subseteq \mathbb{R}^n \) of complexity at most \( (p, q) \), there exists an index \( i \in \{1, \ldots, t\} \) and a semialgebraic homeomorphism \( h : \mathbb{R}^n \rightarrow \mathbb{R}^n \) of complexity at most \( u, v \) such that \( h(S_i) = S \).}
\end{quote}

Using Proposition \ref{pas} and Proposition \ref{sdf}, the statement \(\Psi(n, d, p, t, u)\) can be expressed as a first-order formula with coefficients in \(\mathbb{Z}\) in the theory of real closed fields. Since \(\mathbb{R} \models \Psi(n, d, p, t, u)\), by the Tarski Tranfer Principle, Proposition \ref{Principle}, \(R \models \Psi(n, d, p, t, u)\), which prove the first point of the Theorem.
Using the decidability of the first-order theory of real closed fields, Theorem \ref{decidability}, in  a similar process as in the final part of the proof of Theorem \ref{main}, we establish the computability of the integers \( t = t(n, p, q), \, u = u(n, p, q), \) and \( v = v(n, p, q) \) in terms of \( n, p, \) and \( q \).

This concludes the proof.
\end{proof}

\end{document}